\documentclass[11pt, oneside]{article} 

\usepackage{geometry}                		
\usepackage{graphicx}				
				
\usepackage{subfigure}							 
\usepackage{multirow}
\usepackage{amsmath,amssymb,mathtools}
\usepackage{amsthm}
\usepackage{booktabs}

\usepackage{algpseudocode}
\usepackage{algorithm}
\usepackage{float}
\usepackage{multirow}
\usepackage{tikz,pgfplots}
\usepackage{textcomp}

\newtheorem{theorem}{Theorem}[section]
\newtheorem{remark}[theorem]{Remark}

\newcommand{\inner}[2]{\left\langle #1,#2 \right\rangle}
\newcommand{\norm}[1]{\left\lVert#1\right\rVert}

\newcommand{\innert}[2]{\langle #1,#2 \rangle}
\newcommand{\normt}[1]{\lVert#1\rVert}

\newcommand{\rmd}{\mathrm{d}}

\newcommand{\bE}{\mathbb{E}}

\newcommand{\bN}{\mathbb{N}}

\newcommand{\bP}{\mathbb{P}}

\newcommand{\bR}{\mathbb{R}}

\newcommand{\fR}{\mathfrak{R}}

\newcommand{\cA}{\mathcal{A}}

\newcommand{\cI}{\mathcal{I}}

\newcommand{\cO}{\mathcal{O}}
\newcommand{\cP}{\mathcal{P}}

\newcommand{\cR}{\mathcal{R}}

\title{Weakly intrusive low-rank approximation method \\
  for nonlinear parameter-dependent equations}
\author{Loic Giraldi\thanks{Division of Computer, Electrical and Mathematical Sciences and Engineering, King Abdullah University of Science and Technology, Thuwal, Saudi Arabia. E-mail: {loic.giraldi@kaust.edu.sa}.}~ and   Anthony Nouy\thanks{Ecole Centrale de Nantes, Department of Computer Science and Mathematics, 
       LMJL, CNRS UMR 6629, Nantes, France. Email: anthony.nouy@ec-nantes.fr.} \thanks{This research was supported by the French National Research Agency (grant ANR CHORUS MONU-0005).}}

\date{}

\begin{document}
\maketitle

\begin{abstract}
   This paper presents a weakly intrusive strategy for computing a low-rank approximation of the solution of a  system of nonlinear parameter-dependent equations. The proposed strategy relies on a Newton-like iterative solver which only requires evaluations of the residual of the parameter-dependent  equation and of a preconditioner (such as the differential of the residual)  for instances of the parameters independently.  The algorithm provides an approximation of the set of solutions associated with a possibly large  number of instances of the parameters, with a computational complexity which can be orders of magnitude lower than when using the same Newton-like  solver for all instances of the parameters.  The reduction of complexity  requires efficient strategies for obtaining low-rank approximations of the   residual, of the preconditioner, and of the increment at each iteration of   the algorithm. For the approximation of the residual and the preconditioner, weakly intrusive variants of the empirical interpolation method are introduced, which require  evaluations of entries of the residual and the preconditioner.  Then, an approximation of the increment  is obtained by using a greedy algorithm for low-rank approximation, and a   low-rank approximation of the iterate is finally obtained by using a  truncated singular value decomposition. When the preconditioner is the differential of the residual, the proposed algorithm is interpreted as an inexact Newton solver for which a detailed convergence analysis is provided.   Numerical examples illustrate the efficiency of the method.\end{abstract}


\section{Introduction}
\label{sec:introduction}

The purpose of this paper is to propose weakly intrusive variants of model order reduction methods for the efficient  solution
of a system of nonlinear equations
\begin{equation}
  \label{eq:residual_equation}
  R(u(\xi);\xi) = 0
\end{equation}
whose solution $u(\xi)\in \mathbb{R}^N$ depends on parameters
$\xi$ taking values in a finite set $\Xi$. The set $\Xi$ is here supposed to be given and fixed. It depends on the purpose of the parameter-dependent analysis. It may be a set of random samples (e.g., for statistical learning), a set of interpolation points (e.g. sparse grids), a set of integration points... The proposed approach is then complementary to approximation or integration methods for parameter-dependent functionals.
The parameter-dependent solution is assumed to admit an accurate approximation   of the form
\begin{equation*}
 u(\xi) \approx \sum_{i=1}^r v_i \lambda_i(\xi),
\end{equation*}
where the set of parameter-independent vectors $ v_1, \hdots,v_r$ constitutes a reduced basis in $\mathbb{R}^N$. When identifying $u$ with a tensor in $\mathbb{R}^N \otimes \mathbb{R}^{\#\Xi}$, this can be interpreted as a rank-$r$ approximation of $u$.
Model order reduction methods are usually classified as intrusive if numerical codes for parameter-independent equations can not be used as pure
black-boxes. In~\cite{Giraldi2014} and~\cite{Giraldi2015}, the authors consider the solution of stochastic nonlinear equations with a stochastic Galerkin method, usually qualified as intrusive.
The notion of intrusiveness was relaxed by allowing
the access to pointwise evaluations of the
residual of the equation, therefore resulting in a non (or say weakly) intrusive implementation of stochastic Galerkin methods.  Here, we adopt a similar point of view.

We assume that we have a numerical code which for a given instance of $\xi$  generates a sequence
of approximations $(u^k(\xi))_{k\ge 1}$ converging to $u(\xi)$, and we further assume that we have
access to more or less detailed information from this numerical code.  More precisely,  we consider
Newton-type iterations \begin{equation} \label{eq:newton_solver} u^{k+1}(\xi) =  u^k(\xi) +
P(u^k(\xi);\xi)^{-1} R(u^k(\xi);\xi), \end{equation} and we assume that we have access to
evaluations of the residual $R(u^k(\xi);\xi)\in \mathbb{R}^N$ and the preconditioner
$P(u^k(\xi);\xi)\in \mathbb{R}^{N\times N}$ (such as the differential of the residual), or some of
their entries.  A classical approach consists in using the iterative
algorithm~\eqref{eq:newton_solver} for each instance of $\xi$ independently. Here, we formally
apply the iterative algorithm for all values of $\xi$ simultaneously and introduce an additional
truncation step  in order to generate a sequence of low-rank iterates.  The approach is similar to
truncated iterative methods introduced for the solution for tensor-structured linear equations
in~\cite{Bachmayr2015, Ballani2013, Giraldi2014a, Kressner2011a}. The resulting algorithm takes the
form \begin{equation} u^{k+1}(\cdot) =  \Pi_\varepsilon( u^k(\cdot) + \widetilde
P(u^k(\cdot);\cdot)^{-1} \widetilde R(u^k(\cdot);\cdot)), \end{equation} where $\widetilde
R(u^k(\cdot);\cdot)$ and $\widetilde P(u^k(\cdot);\cdot)$ are low-rank approximations of $
R(u^k(\cdot);\cdot)$ and $ P(u^k(\cdot);\cdot)$, and where $\Pi_\varepsilon$ is a truncation
operator such that $\Pi_\varepsilon(v)$ provides a low-rank approximation of a function $v$
with a controlled precision $\varepsilon$. The algorithm provides a low-rank approximation of the
solution of the parameter-dependent equation without using any snapshot of the solution. Assuming
that the residual and the preconditioner admit accurate approximations with a low rank, a limited
information on these quantities (i.e.\ a small number of evaluations of their entries) is
sufficient to construct these approximations, which can yield a significant reduction of complexity
when compared to a classical approach. Here, we rely on variants of the empirical interpolation
method (EIM)~\cite{Barrault2004} for the construction of these approximations.

In contrast to~\cite{Kunisch2002, Chaturantabut2010, Negri2015}, the result of the method is not a
reduced order model which is then evaluated in an online phase, but an approximation of the
solution of a possibly large set of samples.  However, if the samples are a set of quadrature
points, interpolation points, or random samples, standard integration, interpolation or
least-squares methods can then be used to obtain a representation of the parameter-dependent
solution in a suitable approximation format.  \bigskip

The paper is organized as follows.  In Section~\ref{sec:appr-resid-prec}, we consider the
approximation of the residual  in the particular case where for a given $v(\xi)$ with low-rank
representation, we have a partial knowledge on the low-rank representation of the residual
$R(v(\xi);\xi)$.  In this case, we introduce a variation of the approach proposed
in~\cite{Casenave2014} in order to compute a low-rank representation of this residual with a
rigorous control of the error. The same approach can be used for obtaining an approximation of the
preconditioner using a partial knowledge on its low-rank representation. In
Section~\ref{sec:appr-struct-resid}, we consider the approximation of the residual and the
preconditioner without a priori knowledge on their representations as parameter-dependent algebraic
quantities, and we propose an approximation method which requires simple evaluations of entries of
these quantities. The approach relies on the EIM for vector- or matrix-valued parameter-dependent
functions (see e.g.~\cite{Negri2015} for the matrix-valued case), and includes a statistical
control of the error.  Note that the proposed approach differs from  the discrete EIM proposed
in~\cite{Chaturantabut2010} in that it does not require the evaluations of samples of the solution
to compute a reduced basis for its low-rank representation, and it includes a rigorous control of
the error.  In Section~\ref{sec:comp-incr}, we introduce a greedy rank-one algorithm (see e.g.
\cite{Cances2011,Falco2011}) for computing an approximation of $\widetilde P(u^k(\xi);\xi)^{-1}
\widetilde R(u^k(\xi);\xi)$ which  exploits the low-rank structure of the operator $\widetilde
P(u^k(\xi);\xi)$ and right-hand side $\widetilde R(u^k(\xi);\xi)$.   In
Section~\ref{sec:trunc-newt-solv}, we present the Newton-like truncated solver and we analyze its
convergence in the particular case of a standard Newton truncated solver, which is interpreted as
an inexact Newton algorithm~\cite{Dembo1982}. In Section~\ref{sec:numerical-ex}, numerical examples
illustrate the efficiency of the method.

\section{Approximation of residual and preconditioner with partially known
  low-rank structure}
\label{sec:appr-resid-prec}

In this section, we consider the approximation of the residual
$R(\xi) := R(u(\xi);\xi)$ and of the preconditioner $P(\xi):=
P(u(\xi);\xi)$
for a given $u(\xi)$. It is assumed that when $u(\xi)$ admits a given representation of the form
$u(\xi) = \sum_{i=1}^m v_i \lambda_i(\xi)$, the
residual and the preconditioner also admit representations of the form
\begin{equation}
  \label{eq:residual_exact_aff_dec}
  R(\xi) = \sum_{i=1}^{s} g_i\gamma_i(\xi) \quad \text{and} \quad P(\xi) =
  \sum_{i=1}^{p} F_i \phi_i(\xi),
\end{equation}
where the vectors $g_i \in \bR^N$ and matrices $F_i \in \bR^{N\times N}$ are not known but where the  real-valued functions $\gamma_i(\xi)$ and $\phi_i(\xi)$ are known.

We follow~\cite{Casenave2014} in order to construct an approximation of $R(\xi)$ and $P(\xi)$ based
on the knowledge of $\gamma(\xi):= (\gamma_i(\xi))_{i=1}^s$ and $\phi(\xi) :=
(\phi_i(\xi))_{i=1}^p$ and  a minimal number of evaluations of $R(\xi)$ and $P(\xi)$ at some
suitable points in $\Xi$. Note that the knowledge of vectors $\{g_i\}_{i=1}^{s}$ and matrices
$\{F_i\}_{i=1}^p$ is not required, hence this \emph{weakly intrusive} denomination. Here, the
novelty lies in a rigorous control of the error.  The strategy is presented for the approximation
of the residual. The application to the approximation of the preconditioner is straightforward.
\bigskip

Let us assume that an interpolation $\cI_r[\gamma](\xi) $ of $\gamma(\xi)$ is
available in the form
\begin{equation}
 \cI_r[\gamma](\xi) = \sum_{j=1}^r \gamma(\xi^\star_j) \alpha_j(\xi), \label{Irgamma}
\end{equation}
where the $\xi^\star_j$ are some interpolation points in $\Xi$ and the $\alpha_j(\xi)$ are real-valued functions satisfying the interpolation property
\begin{align}
\alpha_j(\xi_i^\star) = \delta_{i,j} \quad \text{for all} \quad 1\le i,j\le r.\label{alpha_interpolation_property}
\end{align}
We then obtain an approximation $\cI_r[R](\xi)$ of the residual $R(\xi)$ of the form
\begin{align*}
  \cI_r[R](\xi) = \sum_{i=1}^{s} g_i \cI_r[\gamma]_i(\xi) =  \sum_{i=1}^{s} g_i  \sum_{j=1}^r \gamma_i(\xi^\star_j) \alpha_j(\xi) =   \sum_{j=1}^r R(\xi_j^\star) \alpha_j(\xi) ,
\end{align*}
which is an interpolation of $R(\xi)$ at points $\{\xi_j^\star\}_{j=1}^r$.
Let $\Vert \cdot \Vert$ be a norm in $\bR^N$ associated with an inner product
$\innert{\cdot}{\cdot}$. The interpolation error on the residual is
\begin{align}
  \normt{R(\xi)-\cI_r[R](\xi)} =  \normt{\gamma(\xi)-\cI_r[\gamma](\xi)}_W,\label{eq:eim_err_res}
\end{align}
where $W  = (\innert{g_i}{g_j})_{1\le i,j\le s} \in \bR^{{s}\times {s}}$ is the Gram matrix of the set of vectors
$\{g_i\}_{i=1}^{s}$, and
$\normt{\cdot}_W$ is the semi-norm in $\bR^s$ induced by $W$, defined by $\Vert x \Vert_W^2 = x^T
Wx$. Therefore, in order to obtain a sharp control of the error of interpolation  of $R(\xi)$, the
error of interpolation of $\gamma$ has to be controlled with respect to the semi-norm $\Vert
\cdot\Vert_W$ and not the standard Euclidean norm in $\bR^s$. We will then propose a mean to
compute the Gram matrix $W$ with less than $r$ evaluations of the residual $R(\xi)$, and an
empirical interpolation method for the construction of an interpolation $\cI_r[\gamma]$ controlled
with respect to the semi-norm $\Vert \cdot \Vert_W$.

\subsection{Computation of the Gram matrix}
\label{sec:exact-eval-g}

Let $\Xi = \{\xi_k\}_{k=1}^Q$ and assume $s\le Q$.
The Gram matrix $W$ of the set of vectors $\{g_i\}_{i=1}^{s}$ is equal to
$$
W = G^T MG,
$$
where $G \in \bR^{N\times s}$ is the matrix whose columns are the vectors $\{g_i\}_{i=1}^{s}$, and where $M \in \bR^{N\times N}$ is the symmetric positive definite matrix associated with the chosen residual norm $\Vert \cdot \Vert$ in $\bR^N$. Therefore, it remains to compute the matrix $G$.
Let $\Gamma \in \bR^{s\times Q}$ and $\fR\in\bR^{N\times Q}$ be the matrices whose columns are the evaluations of $\gamma(\xi)$ and $R(\xi)$ respectively, i.e.
\begin{equation*}
  \Gamma  = [\gamma(\xi_1),\hdots,\gamma(\xi_Q)]  \quad
  \text{and} \quad \fR = [R(\xi_1),\hdots,R(\xi_Q)],
\end{equation*}
such that
$$
\fR = G \Gamma$$
holds. If the rank of $\Gamma$ is not $s$, then we can find a factorization $\Gamma = L \widetilde
\Gamma$ where the matrix $\widetilde \Gamma \in \bR^{\tilde s \times Q}$ has full rank $\tilde s<s$
(e.g.\ using SVD or QR factorization) and write $\fR = \widetilde G  \widetilde \Gamma$, with  $\widetilde G = GL$. Therefore, without loss of generality, we now assume that $\Gamma$ has a rank $s$.

The rank of $\Gamma$ being $s$, we can find a sample $\{\xi_i'\}_{i=1}^s$ such that the vectors $\{\gamma(\xi_i')\}_{i=1}^s$ are linearly independent.
Let $\Gamma'\in \bR^{{s}\times
  {s}}$ (resp. $\fR' \in
  \bR^{N\times {s}}$) be the submatrix of
$\Gamma$ (resp. $\fR$) associated
with these samples,
\begin{align*}
  \Gamma' &= [\gamma(\xi'_1), \hdots, \gamma(\xi'_{s})]   \quad
  \text{and} \quad \fR' = [R(\xi'_1),\hdots,R(\xi'_{s})].
\end{align*}
$\Gamma'$ is thus invertible, and
\begin{equation*}
  G = \fR' \left(\Gamma'\right)^{-1},
\end{equation*}
so that $s$ evaluations of $R(\xi)$ are sufficient to compute $G$, and then the Gram matrix $W = G^T M G$.

\subsection{Empirical interpolation method}
\label{sec:empir-interp-meth}

Here, we present the EIM for the construction of an interpolation $\cI_r[\gamma](\xi)$ of the vector-valued function $\gamma(\xi)$, with a control of the approximation error in the semi-norm $\Vert\cdot\Vert_W$.
The interpolation $\cI_r[\gamma](\xi)$ has the form~\eqref{Irgamma}, where the functions  $\alpha_j(\xi)$  are defined
such that for any $\xi\in \Xi$, $(\cI_r[\gamma])_i(\xi) = \gamma_i(\xi)$ for a collection of indices $\{i_j\}_{j=1}^r$, i.e.
\begin{equation}
  \label{eq:eim_reduced_problem}
   \sum_{j=1}^r \gamma_i(\xi^\star_j) \alpha_j(\xi)= \gamma_i(\xi) , \quad \forall i \in \{i_1, \hdots, i_r\}.
\end{equation}
For the selection of the interpolation points and indices, we use a greedy algorithm~\cite{Barrault2004} which generates a sequence of pairs
 $\{(\xi_r^\star,i_r)\}_{r\ge 1}$ defined recursively by
\begin{align}
  \xi_{r+1}^\star &\in \arg\max_{\xi\in\Xi}
  \normt{\gamma(\xi)-\cI_r[\gamma](\xi)}_W,\notag \\
  \label{eq:struct_i_r_range}
  \text{and} \quad i_{r+1} &\in \arg\max_{i \in \{1,\hdots,s\}}
                             |\gamma_i(\xi_{r+1}^\star)-(\cI_r[\gamma])_i(\xi_{r+1}^\star)|,
\end{align}
where for $r=0$, we use the convention $\cI_0[\gamma]=0$. The algorithm is stopped when $\cI_r[\gamma] = \gamma$, which occurs for some $r \le \min\{s,Q\}$.
This construction ensures that the linear system of equations~\eqref{eq:eim_reduced_problem}
is invertible for any $\xi \in \Xi$, and in particular, it ensures that the interpolation
property~\eqref{alpha_interpolation_property} is satisfied. If the algorithm is stopped when $r$ is
such that
 \begin{equation}
  \label{eq:eim_norm_eps}
  \max_{\xi\in\Xi} \normt{\gamma(\xi)-\cI_r[\gamma](\xi)}_W \le \zeta,
\end{equation}
it yields an interpolation of the residual such that
$\Vert R(\xi) - \cI_r[R](\xi) \Vert \le \zeta$. The cost of the algorithm for $r$ iterations is $O(Qr^4 + s r^3  +Qrs(s+r)).$
\begin{remark}
We emphasize that the standard EIM applied to $R(\xi)$ should have required the evaluation of the residual
 $R(\xi)$ for all $\xi\in \Xi$ ($Q$ evaluations), while the proposed approach
requires  the values of
$\gamma(\xi)$ for all $\xi\in \Xi$ and only $s$ evaluations of the residual, where $s$ is the rank of
$R$.
\end{remark}

\begin{remark}
    The strategy presented in Sections~\ref{sec:exact-eval-g} and~\ref{sec:empir-interp-meth} can
    be directly applied for the interpolation of the preconditioner, with an error control with
    respect to a matrix norm associated with an inner product, such as the Frobenius norm.
    Note that, controlling the error with respect to such a norm does not allow a sharp control of
    the error in subordinate matrix norms.
\end{remark}

\section{Approximation of residuals and preconditioners with unknown
low-rank structures}
\label{sec:appr-struct-resid}

In this section, we consider the approximation of the residual and the preconditioner
without a priori knowledge on their representations as parameter-dependent algebraic quantities.
We assume that if $u(\xi)
= \sum_{i=1}^m v_i \lambda_i(\xi)$, then the residual $R(\xi) := R(u(\xi);\xi)$ and the
preconditioner $P(\xi):=P(u(\xi);\xi)$ are
well-approximated in low-rank format, i.e.
\begin{equation*}
  R(\xi) \approx \sum_{i=1}^r g_i \gamma_i(\xi) \quad \text{and} \quad P(\xi)
  \approx \sum_{i=1}^p F_i\phi_i(\xi),
\end{equation*}
with moderate ranks $r$ and $p$. However, we have no (even partial) information on these low-rank representations.

First, we present the strategy for interpolating the residual.
Then a statistical error bound is derived for the a posteriori control of the approximation
error. Finally, the method is extended to the interpolation of the preconditioner.

\subsection{Interpolation of the residual}
\label{sec:appr-strat}

We here use a randomized version of the EIM, which is called adaptive
cross approximation with partial pivoting in other contexts
\cite{Bebendorf2014}, for the construction of a sequence of interpolations
 of $R(\xi)$ of the form
\begin{equation}
 \cI_r[R](\xi) = \sum_{j=1}^r R(\xi^\star_j) \alpha_j(\xi), \label{IrR}
\end{equation}
where the $\xi^{\star}_j$ are interpolation points in $\Xi$ and the $\alpha_j(\xi)$ are real-valued
functions satisfying the interpolation property
$\alpha_k(\xi^\star_j)=\delta_{j,k}\ ,1\le j,k\le r$. These functions are defined such that
 for any $\xi\in \Xi$, $(\cI_r[R])_i(\xi) = R_i(\xi)$ for a collection of indices $\{i_j\}_{j=1}^r$, i.e.
\begin{equation}
  \label{eq:residual_aca}
   \sum_{j=1}^r R_i(\xi^\star_j) \alpha_j(\xi)= R_i(\xi) , \quad \forall i \in \{i_1, \hdots, i_r\}.
\end{equation}
The strategy differs from the one of Section~\ref{sec:empir-interp-meth} for the selection
 of interpolation points.
 Here, given $\{(i_j,\xi^\star_j)\}_{j=1}^r$ and the corresponding interpolation  $\cI_r[R]$, we select uniformly at random the point $\xi_{r+1}^\star$ in
 $\Xi\setminus \{\xi_k^\star\}_{k=1}^r$. If $R(\xi^\star_{r+1})-\cI_r[R](\xi^\star_{r+1})=0$, then the point is rejected and a new candidate point $\xi^\star_{r+1}$ is randomly generated.  If $R(\xi^\star_{r+1})-\cI_r[R](\xi^\star_{r+1})\neq 0$, an associated index $i_{r+1}$ is selected such that
\begin{equation}
  \label{eq:blind_i_r_range}
  i_{r+1} \in \arg\max_{i \in \{1,\hdots,N\}} | R_i(\xi_{r+1}^\star)-\cI_r[R]_i(\xi_{r+1}^\star) |.
\end{equation}
The selection of interpolation points does not satisfy an optimality condition but in contrast to  standard EIM, it does not require the evaluation of $R(\xi)$ for all $\xi$.
The condition  $R(\xi^\star_{r})-\cI_{r-1}[R](\xi^\star_{r})\neq 0$ ensures that the system of equations~\eqref{eq:residual_aca} admits a unique solution (see~\cite{Bebendorf2014}).
Assuming that the number of rejections is $o(1)$, the cost of the first $r$ iterations of this algorithm is $O(r^4+Nr^2)$.
\begin{remark}
  Note that the cost of interpolation can be drastically reduced when
   the structure of the residual is partially known (see Section~\ref{sec:appr-resid-prec}).
   Equation~\eqref{eq:struct_i_r_range}
  requires the computation of $s$ entries of $\gamma$ (i.e.\
  $i_{r+1}\in\{1,\hdots,s\}$), while Equation~\eqref{eq:blind_i_r_range}
  requires $N$ components of the residual $R$ (i.e.\
  $i_{r+1}\in\{1,\hdots,N\}$).
\end{remark}

\subsection{Statistical error control}
\label{sec:error-control}

In order to certify the approximation, we provide a statistical bound for the error of interpolation  of $R(\xi)$, based on evaluations of some entries of $R(\xi)$. Let $(I_k)_{k \in \bN}$ (resp. $(\xi_k)_{k\in\bN}$)
be independent  random variables with values in $\{1,\hdots,N\}$ (resp. $\Xi$)
following the uniform law. Then the random variables $(X_k)_{k\ge 1}$ defined by
\begin{equation*}
  X_k = NQ(R_{I_k}(\xi_k)-(\cI_r)[R]_{I_k}(\xi_k))^2
\end{equation*}
are independent and identically distributed. By the law of large numbers, the random variable
$Y_M = \frac{1}{M} \sum_{k=1}^M X_k$ converges almost surely to
$\bE(X_k) =
\sum_{\xi\in\Xi}\normt{R(\xi)-\cI_r[R](\xi)}_{2}^2 =  \normt{R-\cI_r[R]}_{F}^2$ as $M\to \infty$, i.e.  $Y_M$ is a convergent and unbiased statistical estimation  of the square of the interpolation error with respect to the Frobenius norm.

Let $\sigma_M^2$ be the statistical estimation of the variance of $X_k$, defined by
\begin{equation*}
\sigma_M^2 = \frac{1}{M-1} \sum_{k=1}^M (X_k-Y_M)^2.
\end{equation*}
The random variable
\begin{equation*}
  \frac{Y_M - \normt{R-\cI_r[R]}_{F}^2}{\frac{\sigma_M}{\sqrt{M}}}
\end{equation*}
converges in law to a random variable $T_M$ having the Student's t-distribution with $M-1$ degrees of
freedom, as $M\to \infty$. Letting $t_{\alpha,M}\ge 0$ be such that $\bP(T_M \le t_{\alpha,M}) = \bP(T_M\ge -t_{\alpha,M}) = 1- \alpha $, and
\begin{align}
e_{M,\alpha}^2 = {Y_M + t_{\alpha,M}
  \frac{\sigma_M}{\sqrt{M}}}, \label{eq:ema}
\end{align}
we then have
\begin{align*}
  &\bP\left(\norm{R-\cI_r[R]}_{F} \le e_{M,\alpha} \right) \underset{M\to \infty}{\longrightarrow} \bP(T_M \ge -t_{\alpha,M}) = 1-\alpha, \end{align*}
which means that  $e_{M,\alpha}$ is an asymptotic upper bound with confidence level $1-\alpha$ for the interpolation error.

\subsection{Interpolation of the preconditioner}
\label{sub:interpo_precond}

We define an interpolation $\cI_r[P](\xi)$ of the operator $P(\xi)$ of the form
\begin{equation}
 \cI_r[P](\xi) = \sum_{k=1}^r P(\xi^\sharp_k) \beta_k(\xi), \label{IrP}
\end{equation}
where the $\xi^\sharp_k$ are interpolation points in $\Xi$ and the   $\beta_k(\xi)$ are real-valued functions satisfying the interpolation property $\beta_k(\xi^\sharp_l)=\delta_{k,l}$, $1\le k,l\le r$.
These functions are defined such that for all $\xi \in \Xi$,  $(\cI_r[P])_\alpha(\xi) = P_\alpha(\xi)$ for a subset of pairs of indices $\cA_r = \{ \alpha_k = (i_k,j_k)\}_{k=1}^r \subset \cA := \{1,\hdots,N\}^2$, i.e.
\begin{equation}
  \label{eq:precond_interp}
  \sum_{k = 1}^r P_{\alpha}(\xi^\sharp_k) \beta_k(\xi) = P_{\alpha}(\xi), \quad \forall
 \alpha \in \cA_r.
\end{equation}
For the selection of the interpolation points and corresponding entries of matrices, we use again a greedy strategy. Given $\{\xi^\sharp_k\}_{k=1}^r$ and $\{ \alpha_k\}_{k=1}^r$, we select $\xi^\sharp_{r+1}$ at random
in $\Xi\setminus \{\xi^\sharp_k\}_{k=1}^r$ (until $P(\xi^\sharp_k) - \cI_r[P](\xi^\sharp_k) \neq 0$) and we determine a corresponding pair of indices $\alpha_{r+1} = (i_{r+1},j_{r+1})$ such that
\begin{align*}
  \alpha_{r+1} \in \arg\max_{\alpha \in \cA } \vert P_{\alpha}(\xi_{r+1}^\sharp)-\cI_r[P]_{\alpha}(\xi_{r+1}^\sharp) \vert.
\end{align*}
A statistical control of the error in the Frobenius norm
can be obtained as in Section~\ref{sec:error-control}, with random variables $X_k$ replaced by
\begin{equation}
    \label{eq:error_stat_precond}
    X_k = N^2Q(P_{A_k}(\xi_k)-\cI_r[P]_{A_k}(\xi_k))^2,
\end{equation}
where $(A_k)_{k\in\bN}$ are independent random variables with values in $\cA $ and with uniform
law. For sparse parameter-dependent matrices $P(\xi)$ such that
$P_\alpha(\xi)=0$ for all $\alpha
\in \cA_0$ and all $\xi\in \Xi$, the random variables $A_k$ are taken uniform on $\cA\setminus
\cA_0$ and $N^2$ in Equation~\eqref{eq:error_stat_precond} is replaced by $\#\cA\setminus\cA_0$.

\section{Computation of the iterates}
\label{sec:comp-incr}

Sections~\ref{sec:appr-resid-prec} and~\ref{sec:appr-struct-resid} provide two alternatives for computing low-rank approximations $ \widetilde R(u^k(\xi);\xi):= R(\xi)$ and $ \widetilde P(u^k(\xi);\xi):=P(\xi)$ of the residual $R(u^{k}(\xi);\xi)$ and preconditioner $ P(u^{k}(\xi);\xi)$ at iteration $k$ of the Newton-type algorithm,
\begin{equation*}
  R(\xi) = \sum_{i=1}^{r_R} R(\xi_i^\star) \alpha_i(\xi) \quad \text{and} \quad
  P(\xi) =   \sum_{i=1}^{r_P} P(\xi_i^\sharp)\beta_i(\xi).
\end{equation*}
Here,
 we present an algorithm which exploits these low-rank representations
 for efficiently computing an approximation of the increment $\Delta u(\xi)$, solution of the
 following equation
 \begin{equation}
     \label{eq:equation_increment}
     P(\xi) \Delta u(\xi) = R(\xi).
 \end{equation}
The proposed algorithm is a greedy rank-one algorithm~\cite{Cances2011,Falco2011} which provides a sequence of approximations $(\Delta u_r(\xi))_{r\ge 1}$ with increasing ranks, defined by
\begin{equation*}
  \Delta u_r(\xi)  = \Delta u_{r-1}(\xi) + w_r \theta_r(\xi),
\end{equation*}
where $\Delta u_0=0$, and where the rank-one correction $w_r \theta_r(\xi)$  is the solution of the optimization problem
\begin{align}
\min_{w\in \bR^N,\theta \in \bR^\Xi} \sum_{\xi \in \Xi} \norm{P(\xi) w \theta(\xi) - R_r(\xi)}_M^2,
\label{optim-rank-one}
\end{align}
where \begin{align*}R_r(\xi) &= R(\xi) - P(\xi) \Delta u_{r-1} \\&= \sum_{i=1}^{r_R} R(\xi_i^\star) \alpha_i(\xi) - \sum_{i=1}^{r_P}\sum_{j=1}^{r-1} P(\xi_i^\sharp)w_j\beta_i(\xi)\theta_j(\xi)
:= \sum_{i=1}^{s} g_i \gamma_i(\xi),\end{align*}
and where the matrix $M$ (possibly parameter-dependent) defines a residual norm.
For the solution of~\eqref{optim-rank-one}, we use an alternating minimization
algorithm which consists in successively
\begin{itemize}
\item minimizing over $w \in \bR^N$, which yields the linear system of equations
$A w = b,
$
where
$$
A = \sum_{\xi\in \Xi} P(\xi)^T M P(\xi) \theta(\xi)^2 =  \sum_{i=1}^{r_P} \sum_{j=1}^{r_P} P(\xi_i^\sharp)^T
  M P(\xi_j^\sharp)(\sum_{\xi\in\Xi} \beta_i(\xi)\beta_j(\xi)\theta(\xi)^2),
$$
and
  \begin{align*}
b = \sum_{\xi\in \Xi} P(\xi)^T M R_r(\xi) \theta(\xi) = \sum_{i=1}^{r_P} \sum_{j=1}^s P(\xi_i^\sharp)^T M g_j (\sum_{\xi\in\Xi}  \beta_i(\xi)\gamma_j(\xi)\theta(\xi)),
\end{align*}
\item minimizing over $\theta \in \bR^\Xi$, which yields
  \begin{equation*}
  \theta(\xi) = \frac{w^T P(\xi)^T M R_r(\xi)}{w^T P(\xi)^T M P(\xi)w}, \quad
  \xi \in \Xi,
\end{equation*}
\end{itemize}
and iterating until convergence.

\begin{remark}
In the case where $P(\xi)$ is symmetric positive definite for all $\xi\in \Xi$, then we can choose
for $M$ the parameter-dependent matrix $M = P(\xi)^{-1}$. The optimization
problem~\eqref{optim-rank-one} defining the rank-one correction $w_r\theta_r(\xi)$ is then equivalent to
\begin{align}
    \min_{w\in \bR^N,\theta\in \bR^\Xi} \sum_{\xi\in \Xi} ( w\theta(\xi))^T P(\xi) w \theta(\xi)  -
    2 \sum_{\xi\in \Xi}  ( w\theta(\xi))^T R_r(\xi).\label{optim-rank-one-sym}
\end{align}
In the alternating minimization algorithm, the minimization over $w$ yields a system of equations $Aw=b$ with
$$
A = \sum_{\xi\in \Xi} P(\xi) \theta(\xi)^2 = \sum_{i=1}^{r_P}  P(\xi_i^\sharp) \sum_{\xi\in\Xi} \beta_i(\xi)\theta(\xi)^2,
$$
and
$$
b = \sum_{\xi\in \Xi} R_r(\xi) \theta(\xi) =  \sum_{j=1}^s g_i  \sum_{\xi\in \Xi} \gamma_j(\xi)\theta(\xi),
$$
and the minimization over $\theta$ yields
$$
\theta(\xi) = \frac{w^T R_r(\xi)}{w^T P(\xi)w}, \quad \xi \in \Xi.
$$

\end{remark}

\section{Truncated iterative solver}
\label{sec:trunc-newt-solv}

The proposed algorithm constructs a sequence of approximations $(u^k)_{k\ge 0}$ as follows, starting with $u^0=0$.
 At iteration $k$, we compute low-rank approximations $\widetilde R(u^k(\xi);\xi)$ and $\widetilde P(u^k(\xi);\xi)$ of $ R(u^k(\xi);\xi)$ and $ P(u^k(\xi);\xi)$ with one of the approaches presented in Sections~\ref{sec:appr-resid-prec} and~\ref{sec:appr-struct-resid}. Then, we compute
a low-rank approximation $\Delta u^{k}(\xi)$ of $\widetilde P(u^k(\xi);\xi)^{-1} \widetilde R(u^k(\xi);\xi))$
 with the greedy low-rank algorithm described in Section~\ref{sec:comp-incr}. Finally,  we define the next iterate by
 \begin{align}
u^{k+1} =  \Pi_\varepsilon( u^k+ \Delta u^{k}),\label{eq:truncated_solver}
\end{align}
where $\Pi_\varepsilon$ is a truncation operator such that $\Pi_\varepsilon(v)$ provides a low-rank approximation of a function $v(\xi)$ with a controlled precision $\varepsilon$ in $L^2$ norm, i.e.
$$
\sum_{\xi\in \Xi}\Vert \Pi_\varepsilon(v)(\xi) - v(\xi) \Vert^2 \le \varepsilon^2 \sum_{\xi\in\Xi}
\Vert v(\xi)\Vert^2,
$$
with a practical implementation relying on SVD\@. The truncation operator allows to avoid a blow-up in the representation ranks of the iterates.

%

Now, we analyze the proposed algorithm in the particular case of a Newton solver, where $P(u(\xi);\xi)=-R'(u(\xi);\xi)$, with $R'(u(\xi);\xi)$ the differential of $R(\cdot;\xi)$ at $u(\xi)$, and analyze the proposed algorithm as an inexact Newton method, following  Dembo et al.
\cite{Dembo1982}.  This will provide us guidelines to avoid unnecessary efforts in the  approximation of the different quantities (residual, preconditioner, increments and iterates). We first rewrite the truncated Newton algorithm in the  space
$(\bR^N )^\Xi $ equipped with the norm $\Vert{\cdot}\Vert$ defined by
$  \Vert{v}\Vert^2 = \sum_{\xi\in\Xi} \Vert{v(\xi)}\Vert^2,
$
where $\normt{v(\xi)}$ is the Euclidean norm of $v(\xi)$.
The parameter-dependent nonlinear system of equations is written
\begin{align*}
  \cR(u) :=  (R(u(\xi);\xi))_{\xi\in\Xi} = 0,
\end{align*}
where $\cR: (\bR^N )^\Xi \to (\bR^N )^\Xi $.
We denote by $\cR'(u)$ the differential of the residual $\cR(\cdot)$ at $u$, such that $\cR'(u)v = (R'(u(\xi);\xi)v(\xi))_{\xi\in \Xi}$ for $v\in (\bR^N)^\Xi$. $\cR'(u)$ is an element of the space of linear operators from $(\bR^N)^\Xi$ to $(\bR^N)^\Xi$, which we equip {with the operator norm
$\Vert M \Vert = \max_{v \in (\bR^N)^\Xi} \Vert M v \Vert/\Vert v\Vert$}.

Then the algorithm
can be rewritten
\begin{align*}
  \widetilde \cR'( u^k)\Delta u^k &= -\widetilde \cR( u^k)
                                              + \widetilde r^k, \\
  u^{k+1} &=  u^k + \Delta u^k + e^k,
\end{align*}
where $\widetilde \cR(u^k)$ and $\widetilde \cR'(u^k)$ are approximations of  $ \cR(u^k)$ and $ \cR'(u^k)$ respectively, $\Delta u^k$ is the approximation of $\widetilde \cR'( u^k)^{-1}\widetilde \cR( u^k)$ computed with the greedy rank-one algorithm, $\widetilde r^k$ the associated residual, and $e^k$  represents the
error related to the truncation step.

In the following, we assume that for all $\xi\in\Xi$,
\begin{itemize}
\item[(A1)] there exists a unique solution $u(\xi)$ to $R(u(\xi);\xi) = 0$,
\item[(A2)] $R(\cdot;\xi)$ is continuously differentiable,
\item[(A3)] $R'(u(\xi);\xi)$ is invertible.
\end{itemize}
These assumptions respectively imply that there exists a unique solution to $\cR(u) = 0$, $\cR$ is continuously differentiable, and $\cR'(u)$ is invertible.

\begin{theorem}\label{th:inexact_rate}
  Assume that
  \begin{itemize}
  \item $u^k$ converges to the solution $u$,
  \item $R(\cdot; \xi)$ is Lipschitz continuous uniformly in
    $\xi$, i.e.\ there exists a constant $C>0$ independent of $\xi$ such that  for all $\xi \in \Xi$,
    \begin{align*}
      \norm{R(v; \xi) -R(w;\xi)} \le C \norm{v-w}, \quad \forall v, w \in
      \bR^N,
    \end{align*}
     \item For $k$ sufficiently large,
     $\cR'(u^k)$ is such that
     $$
    \alpha \Vert v \Vert \le  \norm{\cR'( u^k) v} \le \beta \Vert v \Vert ,\quad  \forall v \in
      \bR^N,
     $$
     for some constants $\alpha,\beta$ independent of $k$,
    \item $\widetilde \cR( u^k)$ (resp. $\widetilde \cR'(
      u^k)$) is an approximation of $\cR( u^k)$ (resp.
      $\cR'( u^k)$) such that
      \begin{equation*}
        \Vert {\widetilde \cR( u^k) - \cR( u^k)} \Vert \le
        \rho_k \quad \text{and} \quad \Vert{\widetilde \cR'( u^k) - \cR'( u^k)} \Vert\le
        \rho'_k.
      \end{equation*}
    \end{itemize}
   If $\rho_k$,
    $\Vert{e^k} \Vert$ and $\Vert{\widetilde r_k}\Vert$ are $o(\Vert{\cR(u^k)}\Vert)$ and $\rho_k'$ is $o(1)$, then $u^k$ converges to $u$ superlinearly. Furthermore, if $\rho_k$,
    $\Vert{e^k} \Vert$ and $\Vert{\widetilde r_k}\Vert$ are $O(\Vert{\cR(u^k)}\Vert^2)$ and
    $\rho_k'$ is $O(\Vert{\cR(u^k)}\Vert)$, then
    the sequence $u^k$ converges with order at least $2$.
\end{theorem}
\begin{proof}
Letting $$s^k = u^{k+1}-u^k = \Delta u^k +e^k,$$
the algorithm can be rewritten as an inexact Newton solver
\begin{align*}
  \cR'(u^k)s^k &= -\cR(u^k) + r^k, \quad
  u^{k+1} = u^k + s^k,
\end{align*}
where  the residual $r^k = \cR'(u^k)s^k + \cR(u^k)$ has the following decomposition
\begin{align*}
r^k 
& = \cR'(u^k)e^k + (\cR'(u^k)- \widetilde \cR'(u^k))\Delta u^k +  \cR(u^k)-\widetilde \cR(u^k) + \widetilde r^k .
\end{align*}
Then
\begin{align*}
\Vert r^k \Vert &\le \Vert  \cR'(u^k) \Vert \Vert e^k \Vert +  \rho_k' \Vert \Delta u^k \Vert + \rho_k + \Vert \widetilde r^k \Vert,
\end{align*}
with
$$
\Vert \Delta u^k \Vert \le  \Vert  \widetilde\cR'(u^k)^{-1}  \Vert (\Vert \widetilde
\cR(u^k) \Vert + \Vert \widetilde r^k \Vert) \le \Vert   \widetilde\cR'(u^k)^{-1} \Vert  (\rho_k + \Vert \cR(u^k) \Vert + \Vert \widetilde r^k \Vert),
$$
where \begin{align*}\Vert  \widetilde\cR'(u^k)^{-1} \Vert &\le \Vert  \cR'(u^k)^{-1} \Vert + \Vert  \widetilde\cR'(u^k)^{-1} \Vert \Vert  \cR'(u^k)^{-1} \Vert \Vert \cR'(u^k)^{-1}  - \widetilde \cR'(u^k)^{-1}  \Vert 
\\
&\le \alpha^{-1} + \alpha^{-1} \rho'_k \Vert  \widetilde\cR'(u^k)^{-1} \Vert  . \end{align*} 
For $k$ sufficiently large, we have $\rho_k'\alpha^{-1} <1$, so that  $\Vert  \widetilde\cR'(u^k)^{-1} \Vert \le \frac{\alpha^{-1}}{1-\alpha^{-1}\rho'_k}$ and 
\begin{equation}
    \label{eq:bound_rk}
\Vert r^k \Vert \le \beta \Vert e^k \Vert +  \frac{\alpha^{-1}\rho_k'}{1-\alpha^{-1}\rho_k'}  (\rho_k + \Vert \cR(u^k) \Vert + \Vert \widetilde r^k \Vert) + \rho_k + (1+\alpha^{-1})\Vert \widetilde r^k \Vert.
\end{equation}
 If $\rho_k$,
    $\Vert{e^k} \Vert$ and $\Vert{\widetilde r_k}\Vert$ are $o(\Vert{\cR(u^k)}\Vert)$ and $\rho_k'$ is $o(1)$, then $\Vert r^k \Vert$ is $o(\Vert{\cR(u^k)}\Vert)$. If
     $\rho_k$,
    $\Vert{e^k} \Vert$ and $\Vert{\widetilde r_k}\Vert$ are $O(\Vert{\cR(u^k)}\Vert^2)$ and
    $\rho_k'$ is $O(\Vert{\cR(u^k)}\Vert)$, then $\Vert r^k \Vert$ is $O(\Vert{\cR(u^k)}\Vert^2)$.
    We then conclude by using~\cite[Th. 3.3]{Dembo1982}.
\end{proof}

Even though we provide guidelines to control the convergence rate of the
Newton algorithm, the computation of $\alpha$, $\beta$ and $\rho_k'$ is not a
simple task. It requires the ability to compute the largest and lowest singular
values of a linear operator on $(\bR^N)^\Xi$, and to ensure that the singular
values of $(\cR'(u^k))_{k\in\bN}$ are bounded by $\alpha$ and $\beta$.

\section{Numerical example}
\label{sec:numerical-ex}

\subsection{Diffusion with nonlinear reaction equation}
\label{sec:diff-nlreac}

\subsubsection{Problem setting}
\label{ssub:problem_setting}

Let $\Omega = (0,1)^2$. We want to solve for all $\xi \in \Xi$
the nonlinear PDE
\begin{align}
    \label{eq:cubic_reac}
    -\Delta u + \frac{\xi}{3} u^3 &= 1 \quad \text{on } \Omega,\\
  u &= 0 \quad \text{on } \partial \Omega,\notag
\end{align}
for $\Xi = (\xi_q)_{q=1}^Q$, a set of $Q=5000$ i.i.d.\ samples is drawn such that $\xi =
\exp(\zeta) - 1$, where the distribution of $\zeta$ is uniform between 0 and 10. The PDE is
discretized with a finite element method where the dimension of the approximation space is
$N=9801$.

Assume that an approximation of the solution $u(\xi) = \sum_{i=1}^m v_i \lambda_i(\xi)$ is
available. The strong form of the residual is
\begin{align*}
  R^{\text{strong}}(u(\xi);\xi)
  &= 1 + \Delta u - \frac{\xi}{3} vu^3 \\
  &= 1 + \sum_{i=1}^m \Delta v_i
    \lambda_i(\xi) - \frac{\xi}{3} \sum_{j=1}^m\sum_{k=1}^m\sum_{l=1}^m
    v_j v_k v_l\lambda_j(\xi)\lambda_k(\xi)\lambda_l(\xi) \\
  &= \sum_{i=1}^{1+m+m^3} \gamma_i(\xi) G_i.
\end{align*}
We can clearly see here that the evaluation of $\gamma$ only requires the knowledge of the
collection $(\lambda_i)_{i=1}^m$ and the structure of the equation. It is therefore computable
without having to sample the collection $(G_i)_{i=1}^{1+m+m^3}$. The considered preconditioner for
this problem is the Jacobian of the residual hence, a Newton solver is used for solving this
discretized equation. Given the low-rank structure of the solution, the preconditioner admits then
an expansion of the form
\begin{align*}
  P(\xi) = \sum_{i=1}^{1+m^2} P_i\phi_i(\xi),
\end{align*}
where $\phi_1(\xi) = 1$ comes from the diffusion term, while
$(\phi_i(\xi))_{i = 2}^{1 + m^2}$ are due to the cubic reaction term and
are of the form $\phi_i(\xi) = \xi \lambda_j(\xi)\lambda_k(\xi)$.

\subsubsection{Computation of the solution by exploiting the known low-rank structure}
\label{ssub:computation_of_the_solution_by_exploiting_the_known_low_rank_structure}

Given $\lambda$, the maps $\gamma$ and $\phi$ are explicitly known. As a consequence, the example
presented in this section fits the framework presented in Section~\ref{sec:appr-resid-prec} and we
are therefore able to solve this nonlinear problem in a weakly-intrusive manner, based on
evaluations of the residual and the preconditioner.

We use here the guidelines provided by Theorem~\ref{th:inexact_rate}. The residual is approximated
by the EIM such that
\begin{equation}
    \label{eq:error_quadratic}
    \norm{\cR(u) - \widetilde \cR(u)} \le \rho_\cR \norm{\cR(u)}^2.
\end{equation}
Note that the error control of the EIM is done according to the supremum norm
$\normt{\cdot}_\infty$ defined by $\normt{v}_\infty = \sup_{\xi\in\Xi}\normt{v(\xi)}$ as stated in
Equation~\eqref{eq:eim_norm_eps}. We therefore use the following
inequality to bound the norm of the residual
\begin{equation*}
    \norm{\cR(u)}^2 \le Q \norm{\cR(u)}_\infty^2.
\end{equation*}
We therefore set the tolerance of the EIM to
\begin{equation*}
    \norm{\cR(u)}_\infty \le \frac{\rho_R}{\sqrt{Q}}\norm{\cR(u)}^2,
\end{equation*}
such that Equation~\eqref{eq:error_quadratic} is satisfied.

Regarding the interpolation of the preconditioner, we arbitrarily set
\begin{equation*}
    \norm{\cP(u)-\widetilde\cP(u)}_{F} = \left(\sum_{\xi\in\Xi} \norm{P(u(\xi);\xi) - \widetilde
    P(u(\xi);\xi)}_F\right)^{1/2} \le \rho_{P} \norm{\cR(u)},
\end{equation*}
where $\normt{\cdot}_F$ denotes the Frobenius norm. Note that this condition implies that the
spectral norm of the error on the interpolation of the Jacobian is $\cO(\normt{\cR(u)})$ as
required in Theorem~\ref{th:inexact_rate}. In practice, $\rho_R$ and $\rho_{P}$ are set
to $10^{-2}$.

Concerning the tolerance parameters of the low-rank linear solver and the SVD truncation,
they are arbitrary set to $10^{-12}$ given that these methods are computationally cheap compared to
the approximation of the residual and the preconditioner. The low-rank solver is both controlled
with respect to the norm of the relative residual and the stagnation of the approximation.

The error estimate $\epsilon$ is given by
\begin{equation}
    \label{eq:error_residual}
  \epsilon(u)^2 = \frac{\sum_{\xi\in\Xi}
  \norm{\widetilde R(u(\xi);\xi)}^2}{\sum_{\xi\in\Xi}\norm{R(0;\xi)}^2},
\end{equation}
and the computational performance of the algorithm is assessed using the cumulative number of calls
to $R$ and $P$.

In Table~\ref{tab:diff-reac-values}, the values of the error indicator $\epsilon(u)$, as well as
the cumulative numbers of calls to $R$ and $P$ are given with respect to the number of iterations
of the global Newton solver. The normalized cost is defined as the ratio between the effective
number of calls to $R$ or $P$ and the number of calls required by a Monte-Carlo method with the entire sample $\Xi$.
First, the quadratic convergence of the Newton's method holds in this numerical experiment. As we
can see, the estimated relative residual goes from $10^{-5}$ to $10^{-10}$ between iterations 4 and
5, in agreement with the convergence rate predicted by Theorem~\ref{th:inexact_rate}. Moreover, the
table illustrates substantial computational gains. In particular, the proposed strategy requires
the assembly of $448$ residuals to solve the problem which corresponds to $1.18\%$ of the assembly
of the $25 000$ residuals requires for performing $5$ iterations for each sample in a Monte-Carlo
approach.  The gain is even more important for preconditioners, as the technique requires the
computation of only $4.40\text{\textperthousand}$ of the number of preconditioner evaluations
required by a Monte-Carlo method.

  Note that the cost of the method is cumulative, but normalized by the computational cost of the Monte-Carlo method that increases as well. As a consequence, the relative cost of computation of the residual (or preconditioner) can decrease between two iterations (e.g., see the cost of the construction of the residual between iterations 3 and 4).

  The number of calls to the residual (or preconditioner) are due to the error estimation and the evaluation of the Gram matrix introduced in Section~\ref{sec:exact-eval-g}. Indeed, $s$ samples are required to estimate the error which is an upper bound of the rank of the approximation of the residual. We can achieve a faster computation of the approximation of the residual by ignoring the structure of the residual if we have access to the computation of one entry of the residual, and using the statistical error control as illustrated in the next section.

\begin{table}[htb]
\centering
\caption{Error indicator, cumulative number of calls to $R$ and $P$ and normalized cost of the
assemblies compared to a Monte-Carlo method w.r.t.\  the number of iterations of the Newton's
solver for the solution of Problem~\eqref{eq:cubic_reac}.}
\label{tab:diff-reac-values}
\begin{tabular}{cccccc}
\toprule
\multirow{2}{*}{Iter.} & \multirow{2}{*}{$\epsilon(u)$} & \multicolumn{2}{c}{Residual}
                       &\multicolumn{2}{c}{Preconditioner} \\
\cmidrule(lr){3-4} \cmidrule(lr){5-6}
    & &\#Calls
    & Cost
    & \#Calls
    & Cost \\ \midrule
1& $2.40\times10^{-1} $ & $3$   & $6.00\times10^{-4}$ & $1$   & $2.00\times10^{-4}$ \\
2& $3.94\times10^{-2} $ & $41$  & $4.10\times10^{-3}$ & $3$   & $3.00\times10^{-4}$ \\
3& $2.27\times10^{-3} $ & $210$ & $1.14\times10^{-2}$ & $18$  & $1.20\times10^{-3}$ \\
4& $1.19\times10^{-5} $ & $375$ & $1.88\times10^{-2}$ & $65$  & $3.25\times10^{-3}$ \\
5& $4.07\times10^{-10}$ & $448$ & $1.18\times10^{-2}$ & $110$ & $4.40\times10^{-3}$ \\ \bottomrule
\end{tabular}
\end{table}

\subsubsection{Approximation without prior knowledge on the structure of the equation}
\label{ssub:approximation_without_prior_knowledge_on_the_structure_of_the_equation}

We consider the problem introduced in Section~\ref{ssub:problem_setting} where we ignore the prior
knowledge on the low-rank structure of the residual and the preconditioner. Therefore, the
strategy introduced in Section~\ref{sec:appr-struct-resid} is considered.

Regarding the tolerance values, the error is set to $10^{-12}$ for the low-rank linear solver and
the SVD truncation. Regarding the randomized EIM, the error is assessed with $M = Q =5000$ entries
and the confidence level is set to $\alpha=95\%$ for the approximation of the residual and the
preconditioner. Let $Z_M$ be defined by
\begin{equation*}
    Z_M^2 = \frac{NQ}{M}\sum_{k=1}^M R_{I_k}(\xi_k)^2,
\end{equation*}
where $(I_k)_k$ and $(\xi_k)_k$ are random variables defined in Section~\ref{sec:error-control}.
Then $Z_M$ is an estimator of $\normt{\cR(u)}$. Therefore, the convergence criterion on the
residual is set such that the algorithm stops when one of the following conditions is satisfied:
\begin{equation}
    \label{eq:error_crit_blind_newton}
    e_{M,\alpha} \le \rho_\cR Z_M^2, \quad \text{or} \quad \max_{1\le k \le
    M} \left| R_{I_k}(\xi_k) - \widetilde R_{I_k}(\xi_k)\right| \le 10^{-15},
\end{equation}
where $(I_k,\xi_k)_k$ are the random entries sampled for the error estimation. The condition on the
supremum norm of the error on the test set avoids excessive tolerances when realizations of
$Z_M^2$ is small. Moreover, the condition $M=Q$ means that the error estimator requires the evaluation of $Q$ entries of the residual. As a consequence, the computational cost of the error estimation is cheaper than a rank-one approximation of the residual.

The approximation of the Jacobian is controlled such that
\begin{equation*}
    e^P_{M,\alpha} \le \rho_{\cR} \norm{\widetilde \cR(u)},
\end{equation*}
where $e^P_{M,\alpha}$ is the upper bound on the error estimated with $M$ entries of the Jacobian
and a confidence level $\alpha$, as derived in Section~\ref{sec:error-control} in the case of the
residual.

Regarding performance measures, we consider both the error estimation $\epsilon(u)$ introduced in
Equation~\eqref{eq:error_residual} and a specific complexity measure. The complexities of the
solution are defined as the ratio of the cumulative number of evaluated entries of $R$ and $P$ and
the cumulative number of entries that should have been evaluated in the case of a Monte-Carlo
method. Note that the measure takes into account the entries used to assess the error and the
sparsity pattern of the preconditioner induced by the finite element as mentioned in
Section~\ref{sub:interpo_precond}.

Table~\ref{tab:diff-reac-blind} shows the error estimation and the complexities with respect to the
iteration of the Newton's solver. First, since the sample $\Xi$ and the initial guess $u^{0}=0$ are
identical to
Section~\ref{ssub:computation_of_the_solution_by_exploiting_the_known_low_rank_structure}, and since
the tolerances are stringent, we observe that the quantity $\epsilon(u)$ has the same convergence
than in Table~\ref{tab:diff-reac-values} and differs only for very small errors. The quadratic
convergence of the Newton method is also satisfied.

One notable difference with the method used in
Section~\ref{ssub:computation_of_the_solution_by_exploiting_the_known_low_rank_structure} is the
computational cost of the approach. While the cost was $1.18\%$ (resp. $4.40\%$) compared to the
Monte-Carlo method in term of residual (resp.\ preconditioner) evaluations, here the normalized
cost is only $6.22\text{\textperthousand}$ for the residual and $1.48\text{\textperthousand}$ for
the preconditioner in terms of entry evaluations.

\begin{table}[htb]
\centering
\caption{Error estimation, complexity and ranks for the solution of
Problem~\eqref{eq:cubic_reac} w.r.t.\ the iterations of the Newton's solver without exploiting the
structure of the residual and the preconditioner.}
\label{tab:diff-reac-blind}
\begin{tabular}{ccccccc}
  \toprule
    \multirow{2}{*}{Iter.} & \multirow{2}{*}{$\epsilon(u)$} & \multicolumn{2}{c}{Normalized cost} & \multicolumn{3}{c}{Rank} \\ 
    & & Residual & Preconditioner & $u$ & $\widetilde \cR$ & $\widetilde \cP $ \\\midrule
    1& $2.40\times10^{-1}$  & $8.08\times10^{-4}$ & $2.11\times10^{-4}$ & 1 & 1 & 1 \\
    2& $3.94\times10^{-2}$  & $1.81\times10^{-3}$ & $3.17\times10^{-4}$ & 7 & 7 & 2 \\
    3& $2.27\times10^{-3}$  & $2.15\times10^{-3}$ & $7.75\times10^{-4}$ & 9 & 10 & 6 \\
    4& $1.20\times10^{-5}$  & $5.64\times10^{-3}$ & $8.99\times10^{-4}$ & 9 & 81 & 5 \\
    5& $3.94\times10^{-10}$ & $6.22\times10^{-3}$ & $1.48\times10^{-3}$ & 7 & 94 & 10 \\
    \bottomrule
\end{tabular}

\end{table}

\subsection{Nonlinear diffusion equation}
\label{sec:nonl-diff}

We are interested now in a nonlinear diffusion equation defined on $\Omega = (0,1)^2$ for all
$\xi\in\Xi$ by
\begin{align}
    \label{eq:exp_diff}
  -\nabla\cdot (\exp(\xi u(\xi)) \nabla u(\xi)) &= 1 \quad \text{on } \Omega, \\
  u(\xi) &= 0 \quad \text{on } \partial \Omega.\notag
\end{align}
The sample $\Xi = (\xi_q)_{q=1}^Q$ is such that $Q = 5000$ and $\xi = \exp(\zeta) - 1$ where
$\zeta$ is drawn according to the uniform distribution between $0$ and $3$. The weak form of the
residual is given by
\begin{equation*}
  \inner{v}{R(u(\xi);\xi)} = \int_\Omega v \rmd x - \int_\Omega \exp(\xi
  u(\xi)) \nabla v \cdot \nabla u~ \rmd x,
\end{equation*}
and the Jacobian is
\begin{equation*}
  \inner{v}{R'(u(\xi),\xi)w} = - \int_\Omega \exp(\xi u) \nabla v \cdot \nabla
  w ~\rmd x - \int_\Omega \xi \exp(\xi u) (\nabla u \cdot \nabla v) w~\rmd x.
\end{equation*}
For the preconditioner, we will only consider the symmetric part of the $R'(u(\xi);\xi)$ (i.e.\ the first of the two terms) in order to improve the efficiency of the low-rank solver and avoid to treat
non-symmetric problems. The global solver is therefore a modified Newton's method. Due to the
exponential term, a low-rank expression of the residual or the preconditioner is not directly
available, we are therefore in the framework presented in Section~\ref{sec:appr-struct-resid}.

The mesh used for the finite element approximation is the same as
Section~\ref{ssub:problem_setting}. Regarding the tolerances, error estimation and complexity
estimation, we use the quantities defined in
Section~\ref{ssub:approximation_without_prior_knowledge_on_the_structure_of_the_equation} with the
difference that $\rho_R = 0.1$, $\rho_P = 0.1$ and that the tolerance on the approximation of the
residual is set such that
\begin{equation*}
    e_{M,\alpha} \le \rho_R Z_M,
\end{equation*}
the difference being that the upper bound is linear with $Z_M$ and not quadratic anymore.
As a consequence, the relative error on the approximation of the residual is of the order of $\rho_R$.

Table~\ref{tab:exp-diff} shows the efficiency of the method in terms of relative residual estimate
$\epsilon(u)$ and normalized complexities. The estimated error is $3.29\times10^{-9}$ after 8
iterations. This time the quadratic convergence does not hold because first the preconditioner is
not exactly the derivative of the residual and then the convergence of the error on the
interpolation of the residual is not quadratic anymore. We are nevertheless able to get a high
accuracy in terms of relative residual (reaching $3.28\times10^{-10}$) with a low computational
cost compared to a Monte-Carlo method. Indeed, the final complexity regarding the computation of
the residual and the preconditioner is similar to the computation of the solution of about 10
samples of the deterministic problem, i.e.\ the highest normalized cost between the residual and the
preconditioner is $2.03\text{\textperthousand}$.

\begin{table}[htb]
\centering
\caption{Error estimation, complexity and ranks for the solution of
Problem~\eqref{eq:exp_diff} w.r.t.\ the iterations of the solver.}
\label{tab:exp-diff}
\begin{tabular}{ccccccc}
    \toprule
    \multirow{2}{*}{Iter.} & \multirow{2}{*}{$\epsilon(u)$} & \multicolumn{2}{c}{Normalized cost} & \multicolumn{3}{c}{Rank} \\ 
    & & Residual & Preconditioner & $u$ & $\widetilde \cR$ & $\widetilde \cP $ \\\midrule
    1 & $1.41\times 10^{-1}$ &  $8.08\times 10^{-4}$ & $2.11\times 10^{-4}$ & 1 & 1 & 1 \\
    2 & $1.75\times 10^{-2}$ &  $7.57\times 10^{-4}$ & $6.34\times 10^{-4}$ & 6 & 2 & 5 \\
    3 & $1.79\times 10^{-3}$ &  $7.40\times 10^{-4}$ & $7.75\times 10^{-4}$ & 8 & 2 & 5 \\
    4 & $1.27\times 10^{-4}$ &  $8.07\times 10^{-4}$ & $8.99\times 10^{-4}$ & 8 & 3 & 6 \\
    5 & $1.14\times 10^{-5}$ &  $7.87\times 10^{-4}$ & $1.10\times 10^{-3}$ & 8 & 2 & 9 \\
    6 & $2.16\times 10^{-6}$ &  $8.74\times 10^{-4}$ & $1.27\times 10^{-3}$ & 8 & 4 & 10 \\
    7 & $2.77\times 10^{-7}$ &  $9.80\times 10^{-4}$ & $1.42\times 10^{-3}$ & 8 & 6 & 8 \\
    8 & $2.58\times 10^{-8}$ &  $1.25\times 10^{-3}$ & $1.48\times 10^{-3}$ & 8 & 7 & 13 \\
    9 & $2.44\times 10^{-9}$ &  $1.19\times 10^{-3}$ & $1.79\times 10^{-3}$ & 8 & 1 & 15 \\
    10& $3.28\times 10^{-10}$ & $1.68\times 10^{-3}$ & $2.03\times 10^{-3}$ & 7 & 10 & 15 \\
    \bottomrule
\end{tabular}

\end{table}

\section{Conclusion}
\label{sec:conclusion}

A framework for solving parameter-dependent nonlinear equations in a weakly intrusive manner is
proposed. The method requires first the fast approximation of the residual and the preconditioner
in order to be efficient. We show here that they can be interpolated in a weakly intrusive manner
thanks to an extensive use of the empirical interpolation method, in its vector or matrix variants. These interpolations
enables the use of an efficient greedy rank-one solver, which is used to compute the increments of
the solution at each iteration. Finally, the current solution is compressed at each iteration in
order to reduce its representation and the entire strategy is illustrated on numerical examples. A
convergence analysis is performed in the particular case of the Newton's solver, and the theory is
validated experimentally. The efficiency of the methods is illustrated on numerical examples.

This work is proof of concept and opens the way to more complex applications, in particular in
nonlinear mechanics. Indeed, the assembly of the residual and the preconditioner for such problems
represents the main part of the computational costs and the strategy proposed in this paper could be
suitable. The algorithm would be then comparable to the one proposed in~\cite{Farhat2014} where a
sparse integration methodology is used to reduced the assembly cost. The robustness of the method
when a large number of parameters is used should also be assessed in future work.

\bibliographystyle{plain}


\end{document}